\documentclass{amsart}
\usepackage{amsfonts}

\newtheorem{theorem}{Theorem}

\newtheorem{lemma}{Lemma}

\newtheorem{proposition}{Proposition}
\newtheorem{remark}{Remark}

\numberwithin{equation}{section}

\begin{document}
\title{Riemannian Submersions Need Not Preserve Positive Ricci Curvature}
\author{Curtis Pro}
\author{Frederick Wilhelm}

\begin{abstract}
If $\pi :M\rightarrow B$ is a Riemannian Submersion and $M$ has positive
sectional curvature, O'Neill's Horizontal Curvature Equation shows that $B$
must also have positive curvature. We show there are Riemannian submersions
from compact manifolds with positive Ricci curvature to manifolds that have
small neighborhoods of (arbitrarily) negative Ricci curvature, but that
there are no Riemannian submersions from manifolds with positive Ricci
curvature to manifolds with nonpositive Ricci curvature.
\end{abstract}

\maketitle

In \cite{O}, O'Neill shows that if $\pi :M\rightarrow B$ is a Riemannian
submersion, then 
\begin{equation*}
\mathrm{sec}_{B}\left( x,y\right) =\mathrm{sec}_{M}\left( \tilde{x},\tilde{y}%
\right) +3\left\vert A_{\tilde{x}}\tilde{y}\right\vert ^{2},
\end{equation*}%
where $\{x,y\}$ is an orthonormal basis for a tangent plane to $B$, $\{%
\tilde{x},\tilde{y}\}$ a horizontal lift of $\{x,y\}$ and $A$ is O'Neill's $%
A $-tensor. An immediate observation is that if the sectional curvature of
all two-planes in $M$ are bounded below by $k$, then the same must be true
for all two-planes in $B$. In particular, Riemannian submersions preserve
positive sectional curvature.

David Wraith asked us for an example of a Riemannian submersion that does
not preserve positive Ricci curvature. While many people think that such an
example must exist, we do not know of one in the literature but provide some
here.

\begin{theorem}
\label{main theorem} For any $C>0$, there is a Riemannian submersion $\pi
:M\rightarrow B$ for which $M$ is compact with positive Ricci curvature and $%
B$ has some Ricci curvatures less than $-C$.
\end{theorem}

The examples are constructed as warped products $S^{2}\times _{\nu }F$,
where $F$ is any manifold that admits a metric with Ricci curvature $\geq 1$%
, and the metrics on $S^{2}$ are $C^{1}$--close to the constant curvature $1$
metric. In his thesis, the first author provides examples where the base
metric on $S^{2}$ can be $C^{1}$--close to any predetermined,
positively-curved, rotationally-symmetric metric on $S^{2},$ and $F$ is any
manifold that admits a metric with Ricci curvature $\geq 1.$ \cite{Pro}

Since the metrics on our base spaces are $C^{1}$--close to the constant
curvature $1$ metric on $S^{2}$ and yet have a curvature that is $\leq -C,$
the region of negative curvature is necessarily small, and in fact has
measure converging to 0 as $-C\rightarrow -\infty .$ In this sense, our
examples are local. On the other hand, we show that a truly global example
is not possible.

\begin{theorem}
\label{No Global Example}If $M$ is a compact Riemannian manifold with
positive Ricci curvature, then there is no Riemannian submersion $\pi
:M\longrightarrow B$ to a space $B$ with nonpositive Ricci curvature.
\end{theorem}

We prove Theorem \ref{main theorem} in Sections 1 and 2 and Theorem \ref{No
Global Example} in Section 3.

We are grateful to Bun Wong for asking a question that lead us to Theorem %
\ref{No Global Example}, to Pedro Sol\'{o}rzano for assisting us with a
calculation, and especially to David Wraith for insightful criticisms of the
manuscript and asking us to provide the examples of Theorem \ref{main
theorem}.

\section{Vertical Warping}

Given a Riemannian submersion $\pi :M\rightarrow B$, the vertical and
horizontal distributions are defined as $\mathcal{V}:=ker\,\pi _{\ast }$ and 
$\mathcal{H}:=(ker\,\pi _{\ast })^{\perp }$, respectively. This gives a
splitting of the tangent bundle as 
\begin{equation*}
TM=\mathcal{V}\oplus \mathcal{H}.
\end{equation*}%
If $g$ is the metric on $M$, we denote by $g^{h}$ and $g^{v}$ the
restrictions of $g$ to $\mathcal{H}$ and $\mathcal{V}$. Define a new metric $%
g_{\nu }:=e^{2\nu }g^{v}+g^{h}$ on $M$, where $\nu $ is any smooth function
on $B$. Note that both $\mathcal{H}$ and $g^{h}$ are unchanged, so $\pi
:(M,g_{\nu })\rightarrow B$ is also Riemannian.

The calculations that give important geometric quantities associated to $%
g_{\nu }$ in terms of $g$ and $\nu$ are carried out in \cite{GromWals} on
page 45. In particular, the $(0,2)$ Ricci tensor $\mathrm{Ric}_{\nu}$ of $%
g_{\nu }$ is given in detail. When $M=B^{m}\times F^{k}$ and $g$ is a
product metric, these quantities reduce to the following (Corollary 2.2.2 
\cite{GromWals}).

For horizontal $X,Y$ and vertical $U,V$, we have 
\begin{eqnarray}
\mathrm{Ric}_{\nu }(X,Y) &=&\mathrm{Ric}_{B}(X,Y)-k(\mathrm{Hess\,}\nu
(X,Y)+g(\nabla \nu ,X)g(\nabla \nu ,Y)),  \label{horz} \\
\mathrm{Ric}_{\nu }(X,U) &=&0,\text{ and}  \label{mixed} \\
\mathrm{Ric}_{\nu }(U,V) &=&\mathrm{Ric}_{F}(U,V)-g(U,V)e^{2\nu }(\Delta \nu
+k|\nabla \nu |^{2}).  \label{vert}
\end{eqnarray}%
(There is a sign error in the analog of Equation \ref{vert} in \cite%
{GromWals}.) We denote fields on $B$ and their horizontal lifts via $\pi
_{1}:B\times F\rightarrow B$ by the same letter. We write $B\times _{\nu }F$
to denote the warped product metric $g_{\nu }$ on $B\times F$.

\section{The Warped Product $S_{\protect\varphi }^{2}\times _{\protect\nu }F$%
}

Choose $\varphi :[0,\pi ]\rightarrow \lbrack 0,\infty )$ so that $S^{2}$
with the metric $g_{\varphi }=dr^{2}+\varphi ^{2}d\theta ^{2}$ is a smooth
Riemannian manifold denoted by $S_{\varphi }^{2}$. Let $\nu :[0,\pi
]\rightarrow R$ be a function on $S_{\varphi }^{2}$ that only depends on $r$%
. Consider the warped product $S_{\varphi }^{2}\times _{\nu }F$ where $%
(F,g_{F})$ is any $k$-dimensional manifold $(k\geq 2)$ with $\mathrm{Ric}%
_{F}\geq 1$. Using the notation $\dot{\nu}=\partial _{r}\nu $, since $\nu $
only depends on $r,$ 
\begin{equation*}
\nabla \nu =\dot{\nu}\partial _{r}.
\end{equation*}%
If $L$ is the Lie derivative, we have, 
\begin{eqnarray*}
2\mathrm{Hess}\,\nu &=&L_{\nabla \nu }g_{\varphi } \\
&=&L_{\dot{\nu}\partial _{r}}g_{\varphi } \\
&=&\dot{\nu}L_{\partial _{r}}g_{\varphi }+d\dot{\nu}dr+drd\dot{\nu} \\
&=&2\dot{\nu}\mathrm{Hess}\,r+2\ddot{\nu}dr^{2}.
\end{eqnarray*}%
Thus the Hessian of $\nu $ is given by

\begin{equation*}
\mathrm{Hess}\,\nu =\ddot{\nu}dr^{2}+\dot{\nu}\varphi \dot{\varphi}d\theta
^{2}.
\end{equation*}%
The Ricci tensor of $S_{\varphi }^{2}$ is given as 
\begin{equation*}
\mathrm{Ric}_{S_{\varphi }^{2}}=-\frac{\ddot{\varphi}}{\varphi }g_{\varphi }.%
\text{ (see \cite{Peter}, p.69)}
\end{equation*}%
Let $\mathrm{Ric}_{\nu }^{h}$ and $\mathrm{Ric}_{\nu }^{v}$ denote $\mathrm{%
Ric}_{\nu }$ restricted to the horizontal and vertical distribution,
respectively. Equation (\ref{horz}) can be written as 
\begin{equation}
-\mathrm{Ric}_{\nu }^{h}=\left[ \frac{\ddot{\varphi}}{\varphi }+k(\ddot{\nu}+%
\dot{\nu}^{2})\right] dr^{2}+\varphi \left[ \ddot{\varphi}+k\dot{\nu}\dot{%
\varphi}\right] d\theta ^{2},  \label{horiz ricci}
\end{equation}%
and equation (\ref{vert}) can be written as 
\begin{equation}
\mathrm{Ric}_{\nu }^{v}=\mathrm{Ric}_{F}-e^{2\nu }(\ddot{\nu}+\frac{\dot{%
\varphi}\dot{\nu}}{\varphi }+k\dot{\nu}^{2})g_{F}.  \label{vert ricci}
\end{equation}

Notice that since $\mathrm{Ric}_{F}\geq 1,$ if $\mathrm{Ric}_{\nu }^{h}$ is
positive, then these equations together with Equation (\ref{mixed}) imply
that $S_{\varphi }^{2}\times _{\nu +\ln \lambda }F$ has positive Ricci
curvature, provided $\lambda $ is a sufficiently small positive constant.

By requiring that $\ddot{\varphi}(p)>0$ for some point $p\in (0,\pi )$, the
projection $\pi _{1}:S_{\varphi }^{2}\times _{\nu }F\longrightarrow
S_{\varphi }^{2}$ is a Riemannian submersion for which the base has points
of negative Ricci curvature. Therefore, to describe a Riemannian submersion
that does not preserve positive Ricci curvature, it suffices to find
functions $\varphi $ and $\nu $ so that

\begin{enumerate}
\item $S_{\varphi }^{2}$ is smooth and has points of negative curvature,
that is, 
\begin{equation*}
\ddot{\varphi}(p)=\eta >0
\end{equation*}%
for some point $p\in (0,\pi )$,

\item $\mathrm{Ric}_{\nu }^{h}>0$, that is, 
\begin{eqnarray*}
\ddot{\nu}+\dot{\nu}^{2} &<&-\frac{\ddot{\varphi}}{k\varphi } \\
\dot{\varphi}\dot{\nu} &<&-\frac{\ddot{\varphi}}{k},\text{ and}
\end{eqnarray*}

\item $\nu $ is constant in a neighborhood of $0$ and $\pi .$
\end{enumerate}

For a $C^{r}$--function $f:R\longrightarrow R,$ we write 
\begin{equation*}
\left\Vert f\right\Vert _{C^{r}}:=\sup_{k\leq r}\sup_{x\in \mathbb{R}%
}\left\{ \left\vert f^{(k)}(x)\right\vert \right\} .
\end{equation*}%
The main idea is embodied in the following Lemma.

\begin{lemma}
\label{incomprehensible}Given any sufficiently small $\eta >0$, $p\in (0,\pi
/4),$ and $\varepsilon >0$, there is a $\delta >0$ and smooth functions $%
\varphi ,\nu :[0,\pi ]\rightarrow \lbrack 0,\infty )$ with the following
properties.

\begin{enumerate}
\item On $[0,\pi ]\setminus (p-\delta ,p+\delta ),$ 
\begin{equation*}
||\varphi -\sin ||_{C^{2}}<\varepsilon ,
\end{equation*}%
and $\ddot{\varphi}$ is maximal at $p$ with $\ddot{\varphi}(p)=\eta >0$.

\item On $[0,\pi ],$ 
\begin{equation*}
||\varphi -\sin ||_{C^{1}}<\varepsilon ,
\end{equation*}%
and for all $r$ in a small neighborhood about the endpoints of $[0,\pi ]$, 
\begin{equation*}
\varphi (r)=\sin (r).
\end{equation*}

\item $\nu $ is constant in a neighborhood of $0$ and $\pi .$

\item On $(p-2\delta ,p+\delta ),$ 
\begin{eqnarray*}
\ddot{\nu}+\dot{\nu}^{2} &<&-\frac{\ddot{\varphi}}{k\varphi }\text{ and} \\
\dot{\nu}\dot{\varphi} &<&-\frac{\ddot{\varphi}}{k}.
\end{eqnarray*}

\item On $[0,\pi ]\setminus (p-2\delta ,p+\delta ),$ 
\begin{equation*}
||\nu ||_{C^{2}}\leq 20\frac{\eta }{k}.
\end{equation*}
\end{enumerate}
\end{lemma}

\begin{proof}
By controlling the size of the interval for which $\ddot{\varphi}$ is very
different from $-\sin $, one can construct a function $\varphi $ that
satisfies $\left( 1\right) $ and $\left( 2\right) ,$ by prescribing $\ddot{%
\varphi}$ and integrating.

Since $||\varphi -\sin ||_{C^{2}}$ is small on $[0,\pi ]\setminus (p-\delta
,p+\delta ),$ (4) is satisfied on $\left[ 0,p-\delta \right] ,$ provided 
\begin{eqnarray}
\nu |_{\left[ 0,p-2\delta \right] } &\equiv &0\text{ and}  \notag \\
\left( \ddot{\nu}+\dot{\nu}^{2}\right) |_{\left( p-2\delta ,p-\delta \right]
} &<&0.  \label{nu inequalities--p-2delta}
\end{eqnarray}

Since $\ddot{\varphi}$ is maximal at $p$ with $\ddot{\varphi}(p)=\eta >0,$
and since $||\varphi -\sin ||_{C^{1}}<\varepsilon ,$ (4) is satisfied on $%
\left[ p-\delta ,p+\delta \right] ,$ provided 
\begin{eqnarray}
\left( \ddot{\nu}+\dot{\nu}^{2}\right) |_{\left[ p-\delta ,p+\delta \right]
} &<&-2\frac{\eta }{kp}\text{ and}  \notag \\
\frac{\dot{\nu}|_{\left[ p-\delta ,p+\delta \right] }}{\sqrt{2}} &<&-2\frac{%
\eta }{k}.  \label{nu inequalities}
\end{eqnarray}%
We require that $\eta $ be small enough so that $\left( \frac{\eta }{k}%
\right) ^{2}<\frac{\eta }{100kp}.$ By prescribing $\ddot{\nu}$ and
integrating$,$ we can then choose $\nu $ so that it satisfies (\ref{nu
inequalities--p-2delta}) on $\left[ 0,p-\delta \right] ,$ (\ref{nu
inequalities}) on $\left[ p-\delta ,p+\delta \right] $, and, in addition, 
\begin{eqnarray*}
\dot{\nu}\left( p-\delta \right) &=&-3\frac{\eta }{k}\text{ and} \\
-4\frac{\eta }{kp} &\leq &\ddot{\nu}|_{\left[ p-\delta ,p+\delta \right] }<0.
\end{eqnarray*}%
We then have the following bound on $\dot{\nu}\left( p+\delta \right) $ 
\begin{eqnarray*}
0 &>&\dot{\nu}\left( p+\delta \right) \\
&=&\dot{\nu}\left( p-\delta \right) +\int_{p-\delta }^{p+\delta }\ddot{\nu}
\\
&\geq &-3\frac{\eta }{k}-2\delta \left( 4\frac{\eta }{kp}\right) .
\end{eqnarray*}

So by choosing $\delta $ to be sufficiently small, we can, by prescribing $%
\ddot{\nu},$ extend $\nu $ to satisfy (5) and (3) on $\left[ p+\delta ,\pi %
\right] .$
\end{proof}

By composing with projection, $\pi _{1}:S^{2}\times F\longrightarrow S^{2},$
we can view $r$ as a function on $S^{2}\times F,$ which by abuse of notation
we also call $r$.

\begin{proposition}
\label{still pos} There is an $\eta >0$ with the following property. For any 
$p\in \left( 0,\frac{\pi }{4}\right) ,$ there is an $\varepsilon \left(
p\right) >0$ so that for $\varphi $ and $\nu $ as in Lemma \ref%
{incomprehensible}, 
\begin{equation*}
\mathrm{Ric}_{\nu }^{h}\geq 1/2
\end{equation*}%
on $r^{-1}\left( [0,\pi ]\setminus (p-2\delta ,p+\delta )\right) .$
\end{proposition}

\begin{proof}
On $r^{-1}\left( [0,\pi ]\setminus (p-2\delta ,p+\delta )\right) ,$ we have $%
||\varphi -\sin ||_{C^{2}}<\varepsilon $ and $||\nu ||_{C^{2}}<20\pi \frac{%
\eta }{k}$. Since $\mathrm{Ric}_{F}\geq 1$, we have $\mathrm{Ric}\left(
S^{2}\left( 1\right) \times F\right) \geq 1.$ Combining these facts with 
\begin{equation*}
\mathrm{Ric}_{\nu }^{h}=-\left[ \frac{\ddot{\varphi}}{\varphi }+k(\ddot{\nu}+%
\dot{\nu}^{2})\right] dr^{2}-\varphi \left[ \ddot{\varphi}+k\dot{\nu}\dot{%
\varphi}\right] d\theta ^{2}
\end{equation*}%
gives the result.
\end{proof}

\begin{proof}[Proof of Theorem \protect\ref{main theorem}]
Let $\eta >0$ be as in Proposition \ref{still pos}. Given $C>0,$ choose $%
p\in \left( 0,\frac{\pi }{4}\right) $ so that 
\begin{equation*}
-\frac{\eta }{2p}<-C.
\end{equation*}%
Choose $\varphi $ and $\nu $ as in Proposition \ref{still pos}. Then 
\begin{equation*}
\mathrm{Ric}_{\nu }^{h} \geq 1/2
\end{equation*}%
on $r^{-1}\left( [0,\pi ]\setminus (p-2\delta ,p+\delta )\right) ,$ and by
Part 4 of Lemma \ref{incomprehensible}, $\mathrm{Ric}_{\nu }^{h}>0$ on $%
r^{-1}(p-2\delta ,p+\delta ).$ It follows that $\mathrm{Ric}_{\nu } >0$
everywhere.

On the other hand, along $r^{-1}\left( p\right) ,$ the Ricci curvature of $%
S_{\varphi }^{2}$ is 
\begin{equation*}
\mathrm{Ric}_{S_{\varphi }^{2}}=-\frac{\ddot{\varphi}\left( p\right) }{%
\varphi \left( p\right) }g_{\varphi }.
\end{equation*}%
Since $\ddot{\varphi}\left( p\right) =\eta $ and $\varphi \left( p\right) $
is almost $\sin \left( p\right) ,$ 
\begin{equation*}
\mathrm{Ric}_{S_{\varphi }^{2}}\leq -\frac{\eta }{2p}<-C
\end{equation*}
as desired.
\end{proof}

\begin{remark}
The proof above exploits the principle of Lemma 3.1 of \cite{PetWilh1} about
which components of the curvature change a lot under a certain type of $%
C^{1} $--small deformation. This principle will be applied in a revised
version \cite{PetWilh2}, where a revised version of Lemma 3.1 of \cite%
{PetWilh1} will appear.
\end{remark}

\section{The Base has a Positive Ricci Curvature}

In this section, we prove Theorem \ref{No Global Example} with an argument
that is similar in spirit to the proof of the main theorem of \cite{OuWilh}.
We use the notation of \cite{O} for the infinitesimal geometry of a
submersion.

Suppose $\pi :M\longrightarrow B$ is a Riemannian submersion with $M$
compact, $\mathrm{Ric}_M>0$ and $\mathrm{Ric}_B\leq 0.$ Exploiting the
compactness of the unit tangent bundle $T^{1}M$ of $M,$ we get Propositions %
\ref{alm periodic} and \ref{unif cont}, stated below. To pose them we need a
metric space structure on $T^{1}M.$ Any metric that induces the topology of $%
T^{1}M$ will serve, the most geometrically relevant is perhaps the Sasaki
metric (see \cite{Sas}).

\begin{proposition}
\label{alm periodic}Let $\left( T^{1}M,\mathrm{dist}\right) $ be the unit
tangent bundle of $M$ with any metric distance that induces the topology
of $T^{1}M$. For any $\varepsilon >0,$ there is a unit speed geodesic $%
\gamma :\left[ 0,l\right] \longrightarrow M$ so that $l\geq 1$ and 
\begin{equation*}
\mathrm{dist}\left( \dot{\gamma}\left( 0\right) ,\dot{\gamma}\left( l\right)
\right) <\varepsilon.
\end{equation*}
\end{proposition}

\begin{proposition}
\label{unif cont}Let $\mathcal{H}^{1}$ be the set of unit horizontal vectors
for $\pi .$ The map $\mathcal{T}:\left( \mathcal{H}^{1},\mathrm{dist}\right)
\longrightarrow R$ defined by%
\begin{equation*}
\mathcal{T}\left( x\right) =\mathrm{Trace}\left( V\longmapsto T_{V}x\right)
\end{equation*}%
is uniformly continuous. Here $T$ is O'Neill's $T$--tensor.
\end{proposition}

\begin{proof}[Proof of Theorem 2]
Let $\gamma $ be a horizontal, almost periodic geodesic as in Proposition %
\ref{alm periodic}. By re-scaling we may assume that $\mathrm{Ric}_M\geq 1.$
Then%
\begin{equation*}
l\leq \int_{0}^{l}\mathrm{Ric\,}\left( \dot{\gamma},\dot{\gamma}\right) .
\end{equation*}%
Let $\left\{ V_{i}\right\} _{i=1}^{k}$ be a vertically-parallel, orthonormal
framing for the vertical distribution along $\gamma ,$ and let $\left\{
E_{j}\right\} _{j=1}^{b}$ be an orthonormal framing for the horizontal
distribution along $\gamma .$ Using O'Neill's vertizonal and horizontal
curvature equations and the hypothesis that $\mathrm{Ric}_{B}\leq 0,$ we get%
\begin{eqnarray}
l &\leq &\int_{0}^{l}\mathrm{Ric\,}\left( \dot{\gamma},\dot{\gamma}\right) 
\notag \\
&=&\sum\limits_{i}\int_{0}^{l}D_{\dot{\gamma}}\left\langle T_{V_{i}}V_{i},%
\dot{\gamma}\right\rangle -\left\vert T_{V_{i}}\dot{\gamma}\right\vert
^{2}+\left\vert A_{\dot{\gamma}}V_{i}\right\vert ^{2}  \notag \\
&&+\sum_{j}\int_{0}^{l}\mathrm{Ric}_{B}\left( \dot{\gamma},\dot{\gamma}%
\right) -3\left\vert A_{\dot{\gamma}}E_{j}\right\vert ^{2}  \notag \\
&\leq &\sum\limits_{i}\left( \left\langle T_{V_{i}}V_{i},\dot{\gamma}%
\right\rangle |_{0}^{l}+\int_{0}^{l}\left\vert A_{\dot{\gamma}%
}V_{i}\right\vert ^{2}\right) +\sum_{j}\int_{0}^{l}-3\left\vert A_{\dot{%
\gamma}}E_{j}\right\vert ^{2}.  \label{int Ricci}
\end{eqnarray}%
Since $\left\{ V_{i}\right\} _{i=1}^{k}$ and $\left\{ E_{j}\right\}
_{j=1}^{b}$ are orthonormal frames,%
\begin{eqnarray*}
\sum_{j}\left\vert A_{\dot{\gamma}}E_{j}\right\vert ^{2}
&=&\sum_{j}\sum\limits_{i}\left\langle A_{\dot{\gamma}}E_{j},V_{i}\right%
\rangle ^{2} \\
&=&\sum_{i}\sum\limits_{j}\left\langle E_{j},A_{\dot{\gamma}%
}V_{i}\right\rangle ^{2} \\
&=&\sum_{i}\left\vert A_{\dot{\gamma}}V_{i}\right\vert ^{2}.
\end{eqnarray*}%
Therefore $\ref{int Ricci}$ gives us%
\begin{eqnarray*}
1 &\leq &l\leq \sum\limits_{i}\left( \left\langle T_{V_{i}}V_{i},\dot{\gamma}%
\right\rangle |_{0}^{l}\right) \\
&=&\sum\limits_{i}-\left\langle V_{i},T_{V_{i}}\dot{\gamma}\right\rangle
|_{0}^{l} \\
&=&\mathrm{Trace}\left( W\longmapsto -T_{W}\dot{\gamma}\right) |_{0}^{l}.
\end{eqnarray*}%
From Propositions \ref{alm periodic} and \ref{unif cont}, we can choose a
horizontal $\gamma :\left[ 0,l\right] \longrightarrow M$ so that 
\begin{equation*}
\mathrm{Trace}\left( W\longmapsto -T_{W}\dot{\gamma}\right) |_{0}^{l}<\frac{1%
}{2},
\end{equation*}%
which gives a contradiction.
\end{proof}

\end{document}